\documentclass[12pt]{article}
\usepackage{geometry}
\usepackage{amssymb}
\usepackage{amsmath}
\usepackage{float}
\usepackage{color}
\usepackage{graphicx}
\usepackage{subfigure}
\usepackage{tabu}
\usepackage{booktabs}
\usepackage{array}
\usepackage{caption}
\usepackage{amsthm}
\usepackage{booktabs}
\usepackage{cite}
\usepackage{indentfirst}
\numberwithin{equation}{section}
\newtheorem{theorem}{\bf Theorem}[section]
\usepackage[noend]{algpseudocode}
\usepackage{algorithmicx,algorithm}
\newtheorem{lemma}{Lemma}[section]

\newtheorem{remark}{Remark}[section]
\newtheorem{Definition}{Definition}[section]
\newtheorem{Corollary}{Corollary}[section]
	\makeatletter
	\@addtoreset{equation}{section}
	\makeatother

\title{Perturbation analysis for t-product based tensor inverse, Moore-Penrose inverse and tensor system}
\author{
	Zhengbang Cao\footnote{School of Mathematical Sciences, Ocean University of China, Qingdao 266100, China.
		E-Mail: {\tt caozhengbang@stu.ouc.edu.cn}},
	Pengpeng Xie\footnote{Corresponding author. School of Mathematical Sciences, Ocean University of China, Qingdao 266100, China.
		E-Mail: {\tt xie@ouc.edu.cn}.
	}
}
\date{}

\begin{document}
	\maketitle
	\begin{abstract}
	This paper establishes some perturbation analysis for the tensor inverse, the tensor Moore-Penrose inverse and the tensor system
    based on the t-product.
    In the settings of structured perturbations, we generalize the Sherman-Morrison-Woodbury (SMW) formula
    to the t-product tensor scenarios. The SMW formula can be used to perform sensitivity analysis for a multilinear system of equations.
	\\ \hspace*{\fill} 
    \\ \hspace*{\fill} 
	\end{abstract}
	\section{Introduction}
\hskip 2em	Applications with tensors 
 have become increasingly prevalent in recent years. An order $m$ tensor can be regarded as a multidimensional array, 
  which takes the form
	\begin{equation*}
		\mathcal{A}=(a_{i_{1}\cdots i_{m}})\in \mathbb{R}^{n_1\times n_2\times \cdots \times n_m}.
	\end{equation*}
In this paper, we mainly focus on tensors of order three. The well-known representations of tensors are the CANDECOMP/PARAFAC \cite{Carroll1970} and Tucker models \cite{Tucker1966}.
The authors in \cite{Kilmer2011,Braman2010} proposed an entirely different setting
of tensor operation known as t-product in which the
familiar tools of linear algebra can be extended to better understand third-order tensors. The 
 t-product 
 has been proved to be a useful tool in many
 areas such as image processing \cite{Martin2013,Tarzanagh2018} and 
 signal processing \cite{Liu2018,Sun2019}.

\hskip 2em  Jin et al. \cite{Jin2017} 
defined the 
Moore-Penrose inverse of a tensor by the t-product. 
Later, the 
Drazin inverse 
was investigated by Miao 
et al. in \cite{Miao2019Jordan}. 
In fact, there has already been a lot of work on generalized 
inverses and their perturbation theory based on the Einstein product.
For example, in \cite{Ma2019} Ma et al. considered the perturbation theory for the Moore-Penrose inverses of tensors via the Einstein product.
Chang \cite{Chang2020} established the Einstein-product based Sherman-Morrison-Woodbury (SMW) formula for tensors.
  However, as far as we know, limited research has been done on the t-product based perturbation analysis except the weighted T-core-EP inverse \cite{Liu2021_2}. 
  Therefore, 
  this paper is devoted to further developing 
  the perturbation results of 
  tensor inverses, Moore-Penrose inverses and tensor system under the t-product. Especially, for better performing
  sensitivity analysis for a multilinear system of equations, we develop the SMW formula for tensors under the t-product.
	
	\hskip 2em This paper is organized as follows. In section 2, we review basic definitions and notations.
Section 3 details perturbation analysis of the tensor inverse, tensor equations and the SMW formula for the tensor inverse.
Perturbation results of the Moore-Penrose inverse and the least squares problem of tensors are presented in section 4.
In section 5,  we conduct sensitivity analysis for a multilinear system of equations by the SMW formula.
Section 6 gives some conclusions.
	
	\section{Preliminaries}
	\hskip 2em 
	Throughout this paper, we follow notations used in \cite{Kilmer2011,Kilmer2013}.
	Third-order tensors, denoted by calligraphic script letters 
	 with real entries 
	 are considered. Capital letters 
	 refer to matrices, and lower case letters 
	 to vectors. The $i$th frontal slice of tensor $\mathcal{A}$ will be denoted by $A^{(i)}$.
	For $\mathcal{A}\in \mathbb{R}^{n_1\times n_2 \times n_3}$, 
	define \texttt{bcirc} as a block circulant matrix of size $n_1n_3\times n_2n_3$
	\begin{equation*}
		\texttt{bcirc}(\mathcal{A})=\begin{bmatrix}
			A^{(1)}&A^{(n_3)}&\cdots&A^{(2)}\\
			A^{(2)}&A^{(1)}&\cdots&A^{(3)}\\
			\vdots&\vdots&\ddots&\vdots\\
			A^{(n_3)}&A^{(n_3-1)}&\cdots&A^{(1)}
		\end{bmatrix}.
	\end{equation*}
The command \texttt{unfold} reshapes a tensor $\mathcal{A} \in \mathbb{R}^{n_1\times n_2 \times n_3}$ into an $n_1n_3\times n_2$ block-column vector (the first block-column of $\texttt{bcirc}(\mathcal{A}))$, while \texttt{fold} is the inverse, i.e., $\texttt{fold}(\texttt{unfold}(\mathcal{A}))=\mathcal{A}$.
	\begin{Definition}{(t-product) \cite{Kilmer2013}}
		Let $\mathcal{A}\in \mathbb{R}^{n_1\times n_2 \times n_3}$ and $\mathcal{B}\in \mathbb{R}^{n_2\times n_4 \times n_3}$. The t-product $\mathcal{A}*\mathcal{B}$ is the tensor $\mathcal{C}\in \mathbb{R}^{n_1\times n_4 \times n_3}$ defined by
		\begin{equation*}
			\mathcal{C}=\mathtt{fold}(\mathtt{bcirc}(\mathcal{A})\cdot \mathtt{unfold}(\mathcal{B})).
		\end{equation*}
	\end{Definition}
	\hskip 2em Note that the t-product reduces to the standard matrix multiplication when $n_3=1$.
	The Discrete Fourier Transformation (DFT) plays a core role in tensor-tensor product. 
	The DFT on $v\in \mathbb{R}^n$, denoted as $\bar{v}$, is given by
	$\bar{v}=F_nv\in \mathbb{C}^n.$
	Here $F_n$ is the DFT matrix 
			$F_n=(\omega_{jk})_{n\times n}$,
		where $\omega_{jk}=e^{{-2(j-1)(k-1)\pi \texttt{i}/n  }}$ with $\texttt{i}=\sqrt{-1}$ and $F_n$ satisfies $F_n^{*}F_n=F_nF_n^{*}=nI_n.$ The block circulant matrix can be block diagonalized by the DFT, i.e.,
		\begin{equation}\label{DFT}
			(F_{n_3}\otimes I_{n_1})\cdot\texttt{bcirc}(\mathcal{A})\cdot(F_{n_3}^{-1}\otimes I_{n_2})=\bar{A},
		\end{equation}
	
	where $\otimes$ denotes the Kronecker product and  
	$\bar{A}=\texttt{diag}(\bar{A}^{(1)},\bar{A}^{(2)},\ldots,\bar{A}^{(n_3)})$.
    By taking the Fast Fourier Transform (FFT) along each tubal scalar of $\mathcal{A}$, $\bar{\mathcal{A}}=\texttt{fold}(\bar{A})=\texttt{fft}(\mathcal{A},[],3)$ and $\mathcal{A}=\texttt{ifft}(\bar{\mathcal{A}},[],3).$

	\begin{Definition}{(identity tensor) \cite{Kilmer2013}}
		The identity tensor $\mathcal{I}\in \mathbb{R}^{n\times n\times n_3}$ is the tensor with 
		$I^{(1)}$ being the $n\times n$ identity matrix, and other frontal slices being zeros.
	\end{Definition}
	\begin{Definition}{(tensor transpose) \cite{Kilmer2013}}
		If $\mathcal{A}\in \mathbb{R}^{n_1\times n_2\times n_3}$, then $\mathcal{A}^{\mathrm{T}}$ is the $n_2\times n_1\times n_3$ tensor obtained by transposing each of the frontal slices and then reversing the order of transposed frontal slices 2 through $n_3$.
	\end{Definition}
	\begin{Definition}{(inverse tensor) \cite{Kilmer2013}}
		An $n\times n\times n_3$ tensor $\mathcal{A}$ has an inverse $\mathcal{B}$, provided that $\mathcal{A}*\mathcal{B}=\mathcal{I}_{nnn_3}$ and $ \mathcal{B}*\mathcal{A}=\mathcal{I}_{nnn_3}.$
	\end{Definition}
\begin{Definition}
The range of tensor $\mathcal{A} \in \mathbb{R}^{n_1\times n_2\times n_3}$
is defined as
\begin{equation*}
	\mathrm{R}(\mathcal{A})=\left\{\mathcal{A}*\mathcal{X}:\mathcal{X}\in \mathbb{R}^{n_2\times 1\times n_3}\right\}.
\end{equation*}
\end{Definition}
\begin{Definition}\cite{Kilmer2011}
	The inner product of $\mathcal{X}\in \mathbb{R}^{n_1\times 1\times n_3}$ and $\mathcal{Y}\in \mathbb{R}^{n_1\times 1\times n_3}$ is defined as
	\begin{equation*}
		\left\langle\mathcal{X},\mathcal{Y}\right\rangle=\mathcal{X}^{\mathrm{T}}*\mathcal{Y},
	\end{equation*}
and if $\left\langle\mathcal{X},\mathcal{Y}\right\rangle=\mathbf{0}$, we call that the tensor $\mathcal{X}$ is orthogonal to $\mathcal{Y}$.
\end{Definition}
	\begin{Definition}{(Moore-Penrose inverse of tensor) \cite{Jin2017}}
			Let $\mathcal{A}\in \mathbb{R}^{n_1\times n_2 \times n_3}$. If there exists a tensor $\mathcal{X}\in \mathbb{R}^{n_2\times n_1\times n_3 }$ such that
		\begin{equation*}		\mathcal{A}*\mathcal{X}*\mathcal{A}=\mathcal{A},\ \mathcal{X}*\mathcal{A}*\mathcal{X}=\mathcal{X},\ {(\mathcal{A}*\mathcal{X})}^{\mathrm{T}}=\mathcal{A}*\mathcal{X},\ {(\mathcal{X}*\mathcal{A})}^{\mathrm{T}}=\mathcal{X}*\mathcal{A},
		\end{equation*}
	then $\mathcal{X}$ is called the Moore-Penrose inverse of the tensor $\mathcal{A}$ and is denoted by $\mathcal{A}^{\dagger}$.
	\end{Definition}
	\begin{Definition}{(multirank)}\cite{Kilmer2013}
		The multirank of tensor $\mathcal{A}$ is the tubal scalar $\rho = \mathrm{multirank}(\mathcal{A})$ such that $\rho^{(i)}$ is the rank of the ith matrix $\bar{A}^{(i)}$.
	\end{Definition}
	\begin{Definition}\cite{Lu2020,Kilmer2013}
	The Frobenius norm and the spectral norm of $\mathcal{A}\in \mathbb{R}^{n_1\times n_2\times n_3}$ are defined as ${\left\|\mathcal{A}\right\|}_F=\sqrt{\sum_{ijk}{\left|a_{ijk}\right|}^2},\ {\left\|\mathcal{A}\right\|}_2={\left\|\mathtt{bcirc}(\mathcal{A})\right\|}_2$.
	\end{Definition}
	\begin{lemma}\cite{Lu2020}\ ${\left\|\mathcal{A}\right\|}_F=\frac{1}{\sqrt{n_3}}{\left\|\bar{A}\right\|}_F,\ {\left\|\mathcal{A}\right\|}_2={\left\|\bar{A}\right\|}_2.$
	\end{lemma}
	By the results above, we 
	summarize several properties of the two norms.
	\begin{theorem}
		 For $\mathcal{A}$ and $\mathcal{B}$ of appropriate size, the following statements hold:\\
		(a) ${\left\|\mathcal{A}*\mathcal{B}\right\|}_F\le {\left\|\mathcal{A}\right\|}_2\cdot{\left\|\mathcal{B}\right\|}_F$, ${\left\|\mathcal{A}*\mathcal{B}\right\|}_F\le {\left\|\mathcal{A}\right\|}_F\cdot{\left\|\mathcal{B}\right\|}_2$.\\
		(b) ${\left\|\mathcal{A}*\mathcal{B}\right\|}_2\le{\left\|\mathcal{A}\right\|}_2\cdot{\left\|\mathcal{B}\right\|}_2$,\
		${\left\|\mathcal{A}*\mathcal{B}\right\|}_F\le \sqrt{n_3}{\left\|\mathcal{A}\right\|}_F\cdot{\left\|\mathcal{B}\right\|}_F$.\\
		(c) ${\left\|\mathcal{A}+\mathcal{B}\right\|}_2\le{\left\|\mathcal{A}\right\|}_2+{\left\|\mathcal{B}\right\|}_2$,\
		 ${\left\|\mathcal{A}+\mathcal{B}\right\|}_F\le{\left\|\mathcal{A}\right\|}_F+{\left\|\mathcal{B}\right\|}_F$,\
		 ${\left\|\mathcal{A}\right\|}_2\le\sqrt{n_3}{\left\|\mathcal{A}\right\|}_F$.
	\end{theorem}
	\begin{proof}
		It can be deduced from Lemma 2.1 and the properties of matrix norm that
		${\left\|\mathcal{A}*\mathcal{B}\right\|}_F=\frac{1}{\sqrt{n_3}}{\left\|\bar{A}\cdot\bar{B}\right\|}_F\le\frac{1}{\sqrt{n_3}}{\left\|\bar{A}\right\|}_F\cdot{\left\|\bar{B}\right\|}_2={\left\|\mathcal{A}\right\|}_2\cdot{\left\|\mathcal{B}\right\|}_F$, and the rest of the proof
follows analogously.
	\end{proof}
	\section{Perturbation for tensor inverses and tensor equations}
		\hskip 2em In this section, we concentrate on the perturbation theory related to the inverses of tensors, and give two classical perturbation theorems. Furthermore, a classical result of structural perturbation of matrix, i.e. Sherman-Morrison-Woodbury identity is also extended to tensors. 
		Following similar approach, we derive perturbation theorems for tensor equations. It is worth noting that these results still apply for matrices.
	
	
	\begin{theorem}
		Let $\mathcal{A} \in\mathbb{R}^{n\times n\times n_3}$ and $\mathcal{B}=\mathcal{A}+\mathcal{E}$ be invertible respectively, then
		\begin{equation*}\label{c}
			\frac{{\left\|\mathcal{A}^{-1}-\mathcal{B}^{-1}\right\|}_F}{{\left\|\mathcal{A}^{-1}\right\|}_F}\le\kappa_1\frac{{\left\|\mathcal{E}\right\|}_2}{{\left\|\mathcal{A}\right\|}_F} \quad and\quad
			\frac{{\left\|\mathcal{A}^{-1}-\mathcal{B}^{-1}\right\|}_2}{{\left\|\mathcal{A}^{-1}\right\|}_2}\le\kappa_2\frac{{\left\|\mathcal{E}\right\|}_2}{{\left\|\mathcal{A}\right\|}_2},
		\end{equation*}
	where $\kappa_1={\left\|\mathcal{A}\right\|}_F\cdot{\left\|\mathcal{B}^{-1}\right\|}_2$ and   $\kappa_2={\left\|\mathcal{A}\right\|}_2\cdot{\left\|\mathcal{B}^{-1}\right\|}_2$.
	\end{theorem}
	\begin{proof}
		Due to the fact that
		\begin{equation}\label{inversechange}
			\mathcal{B}^{-1}-\mathcal{A}^{-1}=\mathcal{A}^{-1}*(\mathcal{A}-\mathcal{B})*\mathcal{B}^{-1}
		\end{equation}
and from Theorem 2.1, we know
	\begin{equation*}
		{\left\|\mathcal{A}^{-1}-\mathcal{B}^{-1}\right\|}_F\le{\left\|\mathcal{A}^{-1}\right\|}_F{\left\|\mathcal{A}-\mathcal{B}\right\|}_2{\left\|\mathcal{B}^{-1}\right\|}_2.
	\end{equation*}
	Thus, ${{\left\|\mathcal{A}^{-1}-\mathcal{B}^{-1}\right\|}_F}/{{\left\|\mathcal{A}^{-1}\right\|}_F}\le\sqrt{n_3}{{\left\|\mathcal{E}\right\|}_2}{{\left\|\mathcal{B}^{-1}\right\|}_F}=\kappa_1{{\left\|\mathcal{E}\right\|}_2}/{{\left\|\mathcal{A}\right\|}_F}.$ A similar procedure can be applied for the spectral norm.
	\end{proof}
	
Since the inverse of tensor after perturbation is involed in the upper bounds of Theorem 3.1, it is not adequate for practical use. In the following, we impove the results by two lemmas below.
	\begin{lemma}
		For $\mathcal{A}\in\mathbb{R}^{n\times n\times n_3}$, ${\left\|(\mathcal{I}-\mathcal{A})^{-1}\right\|}_2\le(1-{\left\|\mathcal{A}\right\|}_2)^{-1}.$
	\end{lemma}
	\begin{proof}
		It is obvious that
		\begin{equation*}
			(\mathcal{I}-\mathcal{A})^{-1}=\mathcal{I}+\mathcal{A}*(\mathcal{I}-\mathcal{A})^{-1}.
		\end{equation*}
Taking the spectral norm of both sides, we get
	\begin{equation*}
		{\left\|(\mathcal{I}-\mathcal{A})^{-1}\right\|}_2\le{\left\|\mathcal{I}\right\|}_2+{\left\|\mathcal{A}*(\mathcal{I}-\mathcal{A})^{-1}\right\|}_2\\
		\le 1+	{\left\|\mathcal{A}\right\|}_2{\left\|(\mathcal{I-A})^{-1}\right\|}_2,
	\end{equation*}
where Theorem 2.1 is exploited.
	\end{proof}
	\begin{lemma}
		Let $\mathcal{A}\in \mathbb{R}^{n\times n\times n_3}$ be invertible, and $\mathcal{E}\in \mathbb{R}^{n\times n\times n_3}$. If ${\left\|\mathcal{A}^{-1}\right\|}_2{\left\|\mathcal{E}\right\|}_2<1,$
		then $\mathcal{B}=\mathcal{A}+\mathcal{E}$ is invertible.
		\begin{proof}
			From the assumption, we have $\|\bar{A}^{-1}\|_2\|\bar{E}\|_2\le1$. According to \cite[Theorem 8.1.2]{Wang2018}, $\bar{B}$ is invertible, which implies that $\mathcal{B}$ is invertible.
		\end{proof}
	\end{lemma}
	\begin{theorem}
		Let $\mathcal{A}\in \mathbb{R}^{n\times n\times n_3}$ be invertible, $\mathcal{E}\in \mathbb{R}^{n\times n\times n_3}$ and $\mathcal{B}=\mathcal{A}+\mathcal{E}$. If ${\left\|\mathcal{A}^{-1}\right\|}_2{\left\|\mathcal{E}\right\|}_2<1,$
	then
	\begin{equation}\label{j}
		{\left\|\mathcal{B}^{-1}\right\|}_F\le\frac{1}{\gamma_1}{\left\|\mathcal{A}^{-1}\right\|}_F,\ {\left\|\mathcal{B}^{-1}\right\|}_2\le\frac{1}{\gamma_2}{\left\|\mathcal{A}^{-1}\right\|}_2,
	\end{equation}
	and
	\begin{equation}\label{k}
		\frac{{\left\|\mathcal{B}^{-1}-\mathcal{A}^{-1}\right\|}_F}{{\left\|\mathcal{A}^{-1}\right\|}_F}\le\frac{\kappa_1}{\gamma_1}\frac{{\left\|\mathcal{E}\right\|}_2}{{\left\|\mathcal{A}\right\|}_F},\ \frac{{\left\|\mathcal{B}^{-1}-\mathcal{A}^{-1}\right\|}_2}{{\left\|\mathcal{A}^{-1}\right\|}_2}\le\frac{\kappa_2}{\gamma_2}\frac{{\left\|\mathcal{E}\right\|}_2}{{\left\|\mathcal{A}\right\|}_2},
	\end{equation}
	where $\kappa_1={\left\|\mathcal{A}\right\|}_F{\left\|\mathcal{A}^{-1}\right\|}_2$, $\gamma_1=1-\kappa_1\frac{{\left\|\mathcal{E}\right\|}_2}{{\left\|\mathcal{A}\right\|}_F}$, and
	$\kappa_2={\left\|\mathcal{A}\right\|}_2{\left\|\mathcal{A}^{-1}\right\|}_2$, $\gamma_2=1-\kappa_2\frac{{\left\|\mathcal{E}\right\|}_2}{{\left\|\mathcal{A}\right\|}_2}$.
	\end{theorem}
	\begin{proof}
From Lemma 3.2, we know that $\mathcal{B}$ is invertible. Hence,
	\begin{equation}\label{666}
		\begin{aligned}
			{\left\|\mathcal{B}^{-1}\right\|}_F&={\left\|(\mathcal{A}+\mathcal{E})^{-1}\right\|}_F
			={\left\|(\mathcal{I}+\mathcal{A}^{-1}*\mathcal{E})^{-1}*\mathcal{A}^{-1}\right\|}_F\\
			&\le{\left\|(\mathcal{I}+\mathcal{A}^{-1}*\mathcal{E})^{-1}\right\|}_2{\left\|\mathcal{A}^{-1}\right\|}_F.
		\end{aligned}
	\end{equation}
	In view of Lemma 3.1,
	\begin{equation}\label{g}
		{\left\|(\mathcal{I}+\mathcal{A}^{-1}*\mathcal{E})^{-1}\right\|}_2\le(1-{\left\|\mathcal{A}^{-1}*\mathcal{E}\right\|}_2)^{-1}
		\le\frac{1}{1-{\left\|\mathcal{A}^{-1}\right\|}_2{\left\|\mathcal{E}\right\|}_2}
		=\frac{1}{\gamma_1}.
	\end{equation}
	Substituting (\ref{g}) into right-hand side of (\ref{666}), the first inequality of (\ref{j}) follows.
 	Moreover,
	\begin{equation*}
			{\left\|\mathcal{B}^{-1}-\mathcal{A}^{-1}\right\|}_F\le{\left\|\mathcal{A}^{-1}\right\|}_2{\left\|\mathcal{A-B}\right\|}_2{\left\|\mathcal{B}^{-1}\right\|}_F\le \kappa_1\frac{{\left\|\mathcal{E}\right\|}_2}{{\left\|\mathcal{A}\right\|}_F}
	 		\frac{1}{\gamma_1}{\left\|\mathcal{A}^{-1}\right\|}_F,
	\end{equation*}
	which 
	leads to the first part of (\ref{k}). The spectral norm case is similar and omitted.\end{proof}
	
	\hskip 2em The next theorem 
	characterizes the perturbation for tensor equations.
	\begin{theorem}
		Suppose $\mathcal{A}\in \mathbb{R}^{n\times n\times n_3}$ is invertible, $\mathcal{M}=\mathcal{A}+\mathcal{E}$ and $\mathcal{X}$ is the solution to equation $\mathcal{A}*\mathcal{X}=\mathcal{B}$, satisfying 
			${\left\|\mathcal{A}^{-1}\right\|}_2{\left\|\mathcal{E}\right\|}_2<1.$
		Then $\mathcal{M}*\left(\mathcal{X+H}\right)=\mathcal{B+K}$ has a unique solution 
		and
		\begin{equation}\label{3.29}
			\frac{{\left\|\mathcal{H}\right\|}_F}{{\left\|\mathcal{X}\right\|}_F}\le\frac{\kappa_1}{\gamma_1}\left(\frac{{\left\|\mathcal{E}\right\|}_2}{{\left\|\mathcal{A}\right\|}_F}+\frac{\sqrt{n_3}{\left\|\mathcal{K}\right\|}_F}{{\left\|\mathcal{B}\right\|}_F}\right),\quad
			\frac{{\left\|\mathcal{H}\right\|}_2}{{\left\|\mathcal{X}\right\|}_2}\le\frac{\kappa_2}{\gamma_2}\left(\frac{{\left\|\mathcal{E}\right\|}_2}{{\left\|\mathcal{A}\right\|}_2}+\frac{{\left\|\mathcal{K}\right\|}_2}{{\left\|\mathcal{B}\right\|}_2}\right),
		\end{equation}
	where $\mathcal{\kappa}_1={\left\|\mathcal{A}^{-1}\right\|}_2{\left\|\mathcal{A}\right\|}_F$,
	$\gamma_1=1-\kappa_1\frac{{\left\|\mathcal{E}\right\|}_2}{{\left\|\mathcal{A}\right\|}_F}$, and $\mathcal{\kappa}_2={\left\|\mathcal{A}^{-1}\right\|}_2{\left\|\mathcal{A}\right\|}_2$,
	$\gamma_2=1-\kappa_2\frac{{\left\|\mathcal{E}\right\|}_2}{{\left\|\mathcal{A}\right\|}_2}$.
	\end{theorem}
	\begin{proof}
		From Lemma 3.2, we know that $\mathcal{M}$ is invertible, so $\mathcal{M}*\left(\mathcal{X+H}\right)=\mathcal{B+K}$ has a unique solution. Next, we only prove the first inequality of (\ref{3.29}). 
		Rewrite $\mathcal{A+E}*\mathcal{X+H}=\mathcal{B+K}$ as $\mathcal{A}*\mathcal{H}=-\mathcal{E}*\mathcal{X}+\mathcal{K}-\mathcal{E}*\mathcal{H}$,
	which yields $\mathcal{H}=-\mathcal{A}^{-1}*\mathcal{E}*\mathcal{X}+\mathcal{A}^{-1}*\mathcal{K}-\mathcal{A}^{-1}*\mathcal{E}*\mathcal{H}.$
	Therefore,
		\begin{equation*}\label{3.33}
			{\left\|\mathcal{H}\right\|}_F\le{\left\|\mathcal{A}^{-1}\right\|}_2{\left\|\mathcal{E}\right\|}_2{\left\|\mathcal{X}\right\|}_F+{\left\|\mathcal{A}^{-1}\right\|}_2{\left\|\mathcal{K}\right\|}_F+{\left\|\mathcal{A}^{-1}\right\|}_2{\left\|\mathcal{E}\right\|}_2{\left\|\mathcal{H}\right\|}_F,
		\end{equation*}
	which, together with $\|\mathcal{A}^{-1}\|_2\|\mathcal{E}\|_2<1$, gives
	\begin{equation*}\label{3.35}
		{\left\|\mathcal{H}\right\|}_F\le\frac{{\left\|\mathcal{A}^{-1}\right\|}_2{\left\|\mathcal{A}\right\|}_F}{1-{\left\|\mathcal{A}^{-1}\right\|}_2{\left\|\mathcal{E}\right\|}_2}\left(\frac{{\left\|\mathcal{E}\right\|}_2{\left\|\mathcal{X}\right\|}_F}{{\left\|\mathcal{A}\right\|}_F}+\frac{{\left\|\mathcal{K}\right\|}_F}{{\left\|\mathcal{A}\right\|}_F}\right).
	\end{equation*}
	 We arrive at (\ref{3.29}) using  ${\left\|\mathcal{B}\right\|}_F={\left\|\mathcal{A}*\mathcal{X}\right\|}_F\le\sqrt{n_3}{\left\|\mathcal{A}\right\|}_F{\left\|\mathcal{X}\right\|}_F$. 
	\end{proof}
	
	\hskip 2em 
	Note that $\kappa_i$'s 
	can be  treated as condition numbers when interpreting the sensitivity of tensor equations 
	to small perturbations on $\mathcal{A}$ and $\mathcal{B}$.

\hskip 2em
Some of the results and proofs above are derived through the identity (\ref{inversechange}), which shows how the inverse changes if the tensor changes. For matrices, a rank-$k$ perturbation to a nonsingular matrix results in a rank-$k$ correction of the inverse. Specifically, the SMW formula gives a convenient expression for the inverse of the matrix $A+UBV$ where $A \in \mathbb{R}^{n\times n},~U \in \mathbb{R}^{n\times k}$ and $V\in \mathbb{R}^{k\times n}$:
\begin{equation}\label{SMW_M}
\left(A+UBV\right)^{-1}=A^{-1}-A^{-1}U\left(B^{-1}+VA^{-1}U\right)^{-1}VA^{-1},
\end{equation}
in which we assume that  $A$, $B$ and $B^{-1}+VA^{-1}U$ are all nonsingular.
In what follows, we prove that a result similar to (\ref{SMW_M}) also works well on characterizing the inverse of a tensor after structured perturbations.
\begin{theorem}(SMW formula for invertible tensors)
	Suppose that $\mathcal{A}\in \mathbb{C}^{n_1\times n_1 \times n_3}$ and $\mathcal{B}\in \mathbb{C}^{n_2\times n_2 \times n_3}$ are invertible. Given  $\mathcal{U}\in \mathbb{C}^{n_1\times n_2 \times n_3}$, $\mathcal{V}\in \mathbb{C}^{n_2\times n_1 \times n_3}$. If $\mathcal{B}^{-1}+\mathcal{V}*\mathcal{A}^{-1}*\mathcal{U}$ is invertible,
then we have
	\begin{equation*}
		(\mathcal{A}+\mathcal{U}*\mathcal{B}*\mathcal{V})^{-1}=\mathcal{A}^{-1}-\mathcal{A}^{-1}*\mathcal{U}*(\mathcal{B}^{-1}+\mathcal{V}*\mathcal{A}^{-1}*\mathcal{U})^{-1}*\mathcal{V}*\mathcal{A}^{-1}.
	\end{equation*}
\end{theorem}
\begin{proof}
Taking FFT along each tubal scalar of $\mathcal{A}$, $\mathcal{U}$, $\mathcal{B}$ and $\mathcal{V}$ respectively, we obtain the corresponding $\bar{A}$, $\bar{U}$, $\bar{B}$ and $\bar{V}$. The assumption ensures the invertibility of $\bar{A}$, $\bar{B}$ and $\bar{B}^{-1}+\bar{V}\bar{A}^{-1}\bar{U}$. The result holds as
	\begin{equation*}
		(\bar{A}+\bar{U}\bar{B}\bar{V})^{-1}=\bar{A}^{-1}-\bar{A}^{-1}\bar{U}(\bar{B}^{-1}+\bar{V}\bar{A}^{-1}\bar{U})^{-1}\bar{V}\bar{A}^{-1},
	\end{equation*}
	which is a direct application of (\ref{SMW_M}).
\end{proof}
	\section{Perturbation for tensor Moore-Penrose inverses and least squares problem}
	\hskip 2em In this section, based on the Moore-Penrose inverse proposed by Jin et al. in \cite{Jin2017}, we will 
	bound the perturbation of the tensor Moore-Penrose inverse and generalize a classical perturbation theorem for the tensor least squares problem.
We first recall an important property of $\mathcal{A}^{\dagger}$.
	Suppose that
	\begin{equation*}
		(F_{n_3}\otimes I_{n_1})\cdot\texttt{bcirc}(\mathcal{A}^{\dagger})\cdot(F_{n_3}^{-1}\otimes I_{n_2})=\bar{B}= \texttt{diag}(\bar{B}^{(1)},\bar{B}^{(2)},\cdots\bar{B}^{(n_3)}),
	\end{equation*}
	\begin{equation*}\label{3.36}
		(F_{n_3}\otimes I_{n_1})\cdot\texttt{bcirc}(\mathcal{A})\cdot(F_{n_3}^{-1}\otimes I_{n_2})=\bar{A}=\texttt{diag}(\bar{A}^{(1)},\bar{A}^{(2)},\cdots,\bar{A}^{(n_3)}),
	\end{equation*}
	then $\left(\bar{A}\right)^{\dagger}=\bar{B},\  \bar{B}^{(i)}=\left({\bar{A}^{(i)}}\right)^{\dagger}.$
The subsequent theorem gives upper bounds on the perturbation for the t-product based  Moore-Penrose inverse under the
two norms.
	\begin{theorem}
		Let $\mathcal{A}\in \mathbb{R}^{n_1\times n_2\times n_3}$, $\mathcal{B}=\mathcal{A+E}$, then
		\begin{equation*}\label{3.39}
			{\left\|\mathcal{B}^{\dagger}-\mathcal{A}^{\dagger}\right\|}_F\le\sqrt{2}\mathrm{max}\left\{{\left\|\mathcal{A}^{\dagger}\right\|}_2^2,{\left\|\mathcal{B}^{\dagger}\right\|}_2^2\right\}{\|\mathcal{E}\|}_F,
		\end{equation*}
		\begin{equation*}\label{3.40}
			{\left\|\mathcal{B}^{\dagger}-\mathcal{A}^{\dagger}\right\|}_2\le\frac{1+\sqrt{5}}{2} \mathrm{max}\left\{{\left\|\mathcal{A}^{\dagger}\right\|}_2^2,{\left\|\mathcal{B}^{\dagger}\right\|}_2^2\right\}{\|\mathcal{E}\|}_2.
		\end{equation*}
	\end{theorem}
	\begin{proof}
		Apply (\ref{DFT}) to $\mathcal{A}^{\dagger},\mathcal{B}^{\dagger},\mathcal{A},\mathcal{B}$ and $\mathcal{E}$ respectively,
        and then we get $\bar{C}=\bar{A}^{\dagger},\bar{D}=\bar{B}^{\dagger},\bar{A},\bar{B}$ and $\bar{E}$ correspondingly.
	By \cite[Theorem 2.9]{Sun2001}, we deduce that
	\begin{equation*}
		\begin{aligned}		{\left\|\mathcal{B}^{\dagger}-\mathcal{A}^{\dagger}\right\|}_F&=\frac{1}{\sqrt{n_3}}{\left\|\bar{C}-\bar{D}\right\|}_F\le\frac{1}{\sqrt{n_3}}\sqrt{2}\mathrm{max}\left\{{\left\|\bar{D}\right\|}_2^2,{\left\|\bar{C}\right\|}_2^2\right\}{\left\|\bar{E}\right\|}_F\\&
			=\sqrt{2}\mathrm{max}\left\{{\left\|\mathcal{A}^{\dagger}\right\|}_2^2,{\left\|\mathcal{B}^{\dagger}\right\|}_2^2\right\}{\|\mathcal{E}\|}_F.
		\end{aligned}
	\end{equation*}
Similar manipulation gives the spectral norm inequality.
\end{proof}
	\hskip 2em 
	By repeated applications of the relationship between $\mathcal{A}^{\dagger}$ and $\bar{A}^{\dagger}$, we get a rank-preserving perturbation result.
%
	\begin{theorem}
		Let $\mathcal{A}\in \mathbb{R}^{n_1\times n_2\times n_3}$, $\mathcal{B}=\mathcal{A+E}$, $\mathrm{multirank}(\mathcal{A})= \mathrm{multirank}(\mathcal{B})=[r_1,r_2,\ldots,r_{n_3}]
	$, then
	\begin{equation}\label{3.45}
		{\left\|\mathcal{B}^{\dagger}-\mathcal{A}^{\dagger}\right\|}_F\le\mu{\left\|\mathcal{A}^{\dagger}\right\|}_2{\left\|\mathcal{B}^{\dagger}\right\|}_2{\left\|\mathcal{E}\right\|}_F,\
		{\left\|\mathcal{B}^{\dagger}-\mathcal{A}^{\dagger}\right\|}_2\le\lambda{\left\|\mathcal{A}^{\dagger}\right\|}_2{\left\|\mathcal{B}^{\dagger}\right\|}_2{\left\|\mathcal{E}\right\|}_2,
	\end{equation}
		where $\mu$ and $\lambda$ are given blow:
	\begin{equation*}\label{3.47}
		\begin{cases}
			\mu=\sqrt{2},\  \lambda=\frac{1+\sqrt{5}}{2},&\sum_{i=1}^{n_3}r_i<min\left\{n_3n_1,n_3n_2\right\}\\
			\mu=1,\ \lambda=\sqrt{2},&\sum_{i=1}^{n_3}r_i=min\left\{n_3n_1,n_3n_2\right\},n_1\neq n_2\\
			\mu=1,\ \lambda=1,&\sum_{i=1}^{n_3}r_i=n_1n_3=n_2n_3.
		\end{cases}
	\end{equation*}
	\end{theorem}
	\begin{proof}
	We only consider the first inequality of (\ref{3.45}). From \cite[Theorem 2.10]{Sun2001},
	\begin{equation*}
		\begin{aligned}
			{\left\|\mathcal{B}^{\dag}-\mathcal{A}^{\dag}\right\|}_F&=
			\frac{1}{\sqrt{n_3}}{\left\|\bar{B}^{\dag}-\bar{A}^{\dagger}\right\|}_F
			\le\frac{1}{\sqrt{n_3}}\mu{\left\|\bar{B}^{\dagger}\right\|}_2^2{\|\bar{A}^{\dagger}\|}_2^2{\left\|\bar{E}\right\|}_F\\
         &=\mu{\|\mathcal{A}^{\dagger}\|}_2{\left\|\mathcal{B}^{\dagger}\right\|}_2{\left\|\mathcal{E}\right\|}_F.
		\end{aligned}
	\end{equation*}
	The proof is complete as
	$\mathrm{rank}(\bar{A})=\sum\limits_{i=1}^{n_3}{r_i}.$\end{proof}
	From Theorem 4.2, a corollary immediately follows.
	\begin{Corollary}
		If $\mathcal{A}$ and $\mathcal{B}=\mathcal{A+E}$ satisfy the conditions in Theorem 4.2, then
		\begin{equation*}
			\frac{{\left\|\mathcal{B}^{\dagger}-\mathcal{A}^{\dagger}\right\|}_F}{{\left\|\mathcal{B}^{\dagger}\right\|}_2}\le\mu\kappa\frac{{\left\|\mathcal{E}\right\|}_F}{{\left\|\mathcal{A}\right\|}_2},\
			\frac{{\left\|\mathcal{B}^{\dagger}-\mathcal{A}^{\dagger}\right\|}_2}{{\left\|\mathcal{B}^{\dagger}\right\|}_2}\le\lambda\kappa\frac{{\left\|\mathcal{E}\right\|}_2}{{\left\|\mathcal{A}\right\|}_2},
		\end{equation*}
	where $\mu$ and $\lambda$ are as in Theorem 4.2, $\kappa={\left\|\mathcal{A}^{\dagger}\right\|}_2{\left\|\mathcal{A}\right\|}_2$.
	\end{Corollary}
	\hskip 2em 
	In a similar fashion, another relative perturbation theorem of the tensor Moore-Penrose inverse can be obtained.
	\begin{theorem}
	Suppose $\mathcal{A}\in \mathbb{R}^{n_1\times n_2\times n_3}$, $\mathcal{B}=\mathcal{A+E}$, $\mathrm{multirank}(\mathcal{A})=\mathrm{multirank}(\mathcal{B})$. If ${\left\|\mathcal{A}^{\dagger}\right\|}_2{\left\|\mathcal{E}\right\|}_2<1$, then
		\begin{equation*}
			{\left\|\mathcal{B}^{\dagger}\right\|}_2\le{\left\|\mathcal{A}^{\dagger}\right\|}_2/\gamma,
			\end{equation*}
and	
	\begin{equation*}
		\frac{{\left\|\mathcal{B}^{\dagger}-\mathcal{A}^{\dagger}\right\|}_F}{{\left\|\mathcal{A}^{\dagger}\right\|}_2}\le\frac{\mu\kappa_1}{\gamma}\frac{{\left\|\mathcal{E}\right\|}_F}{{\left\|\mathcal{A}\right\|}_F},\ \frac{{\left\|\mathcal{B}^{\dagger}-\mathcal{A}^{\dagger}\right\|}_2}{{\left\|\mathcal{A}^{\dagger}\right\|}_2}\le\frac{\mu\kappa_2}{\gamma}\frac{{\left\|\mathcal{E}\right\|}_2}{{\left\|\mathcal{A}\right\|}_2},
	\end{equation*}
	where
	$\kappa_1={\left\|\mathcal{A}^{\dagger}\right\|}_2{\left\|A\right\|}_F,\
		\kappa_2={\left\|\mathcal{A}^{\dagger}\right\|}_2{\left\|A\right\|}_2,\ and\
	 \gamma=1-{\left\|\mathcal{A}^{\dagger}\right\|}_2{\left\|\mathcal{E}\right\|}_2.$
	\end{theorem}
\hskip 2em 
The technique above works well in analyzing the sensitivity of the tensor least squares problem.
\begin{theorem}
Suppose $\mathcal{A}\in \mathbb{R}^{n_1\times n_2\times n_3},\mathcal{B}\in \mathbb{R}^{n_1\times n_4\times n_3}, \tilde{\mathcal{B}}=\mathcal{B}+\mathcal{K},\tilde{\mathcal{A}}=\mathcal{A}+\mathcal{E}$, $\mathcal{X}$ and $\tilde{\mathcal{X}}=\mathcal{X}+\mathcal{H}$ are minimal-Frobenius norm solutions to the tensor least squares problems $\mathrm{min}\{\|\mathcal{A}*\mathcal{X}-\mathcal{B}\|_F\}$ and $\mathrm{min}\{\|\tilde{\mathcal{A}}*\tilde{\mathcal{X}}-\tilde{\mathcal{B}}\|_F\}$ respectively. If $\mathrm{multirank}(\mathcal{A})=\mathrm{multirank}(\tilde{\mathcal{A}}),\|\mathcal{A}^{\dagger}\|_2\|\mathcal{E}\|_2<1$, then
\begin{equation*}
	\|\mathcal{H}\|_F\le\frac{\kappa_1}{\gamma}\left(\frac{\|\mathcal{E}\|_2}{\|\mathcal{A}\|_F}\|\mathcal{X}\|_F+\frac{\|\mathcal{K}\|_F}{\|\mathcal{A}\|_F}+\frac{\kappa_1}{\gamma}\frac{\|\mathcal{E}\|_2}{\|\mathcal{A}\|_F}\frac{\|\mathcal{R}\|_F}{\|\mathcal{A}\|_F}+\sqrt{n_3}\|\mathcal{E}\|_2\|\mathcal{Y}*\mathcal{X}\|_F\right),
\end{equation*}
\begin{equation*}
	\|\mathcal{H}\|_2\le\frac{\kappa_2}{\gamma}\left(\frac{\|\mathcal{E}\|_2}{\|\mathcal{A}\|_2}\|\mathcal{X}\|_2+\frac{\|\mathcal{K}\|_2}{\|\mathcal{A}\|_2}+\frac{\kappa_2}{\gamma}\frac{\|\mathcal{E}\|_2}{\|\mathcal{A}\|_2}\frac{\|\mathcal{R}\|_2}{\|\mathcal{A}\|_2}+\|\mathcal{E}\|_2\|\mathcal{Y}*\mathcal{X}\|_2\right),
\end{equation*}
where $\kappa_i's$ and $\gamma$ are as in Theorem 4.3,  $\mathcal{R}=\mathcal{B}-\mathcal{A}*\mathcal{X},\mathcal{Y}=(\mathcal{A}^{\dagger})^\mathrm T$. Particularly, if $\mathrm{multirank}(\mathcal{A})=\mathrm{multirank}(\tilde{\mathcal{A}})=[n_2,n_2,\ldots,n_2]$, we have
\begin{equation*}
	\|\mathcal{H}\|_F\le\frac{\kappa_1}{\gamma}\left(\frac{\|\mathcal{E}\|_2}{\|\mathcal{A}\|_F}\|\mathcal{X}\|_F+\frac{\|\mathcal{K}\|_F}{\|\mathcal{A}\|_F}+\frac{\kappa_1}{\gamma}\frac{\|\mathcal{E}\|_2}{\|\mathcal{A}\|_F}\frac{\|\mathcal{R}\|_F}{\|\mathcal{A}\|_F}\right),
\end{equation*}
\begin{equation*}
	\|\mathcal{H}\|_2\le\frac{\kappa_2}{\gamma}\left(\frac{\|\mathcal{E}\|_2}{\|\mathcal{A}\|_2}\|\mathcal{X}\|_2+\frac{\|\mathcal{K}\|_2}{\|\mathcal{A}\|_2}+\frac{\kappa_2}{\gamma}\frac{\|\mathcal{E}\|_2}{\|\mathcal{A}\|_2}\frac{\|\mathcal{R}\|_2}{\|\mathcal{A}\|_2}\right).
\end{equation*}
\end{theorem}

\begin{remark}
		Theorems 4.1 and 4.2 give 
		upper bounds on the absolute perturbation for the Moore-Penrose inverse of tensors, while 
		Theorem 4.3 describes the 
		bounds of the relative perturbation. Theorem 4.4 reflects the sensitivity of the tensor least squares problem to perturbation. It 
		is clear that $\kappa_i's$ can be viewed as condition numbers.
\end{remark}
\section{Perturbation analysis for multilinear system based on the SMW formula}
\hskip 2em We have derived the SMW formula for invertible tensors in section 3. To deal with tensors which are not invertible, we generalize the SMW formula for the Moore-Penrose inverse of tensors in this section. Moreover, based on the new established theorem, a sensitivity analysis is also performed for multilinear system by deriving upper bounds for the solution of the multilinear system. We first propose a lemma necessary for the subsequent work.
\begin{lemma}\label{SMW_Lem}
	Let $\mathcal{A}\in \mathbb{R}^{n_1\times n_2 \times n_3}$. For any $\mathcal{U}\in \mathbb{R}^{n_1\times 1 \times n_3}$,
there exists a tensor $\mathcal{X}\in \mathbb{R}^{n_1\times 1 \times n_3}$ which is contained in $\mathrm{R}(\mathcal{A})$ and a tensor $\mathcal{Y}\in \mathbb{R}^{n_1\times 1 \times n_3}$ being orthogonal to $\mathrm{R}(\mathcal{A})$, such that
	\begin{equation*}
		\mathcal{U}=\mathcal{X}+\mathcal{Y}.
	\end{equation*}
\end{lemma}
	\begin{proof}
		Assume $\bar{\mathcal{A}}=\mathrm{fft}(\mathcal{A},[],3)$ and $\bar{\mathcal{U}}=\mathrm{fft}(\mathcal{U},[],3)$, whose $i$th frontal slices are $\bar{A}^{(i)}$ and $\bar{U}^{(i)}$ respectively. For each $\bar{U}^{(i)}$, there exists $\bar{X}^{(i)} \in \mathrm{R}(\bar{A}^{(i)})$ and $\bar{Y}^{(i)} \in \mathrm{R}^\perp (\bar{A}^{(i)})$ such that
		\begin{equation*}\label{deco}
			\bar{U}^{(i)}=\bar{X}^{(i)}+\bar{Y}^{(i)}.
		\end{equation*}
	Set
	\begin{equation*}
		\mathcal{\bar{X}}=\mathtt{fold}\left(\begin{bmatrix}
			\bar{X}^{(1)}\\\bar{X}^{(2)}\\\vdots\\ \bar{X}^{(n_3)}
		\end{bmatrix}\right), \
	\mathcal{\bar{Y}}=\mathtt{fold}\left(\begin{bmatrix}
		\bar{Y}^{(1)}\\\bar{Y}^{(2)}\\\vdots\\ \bar{Y}^{(n_3)}
	\end{bmatrix}\right),
	\end{equation*}
	and
	\begin{equation*}
		\mathcal{X}=\mathtt{ifft}\left(\bar{\mathcal{X}},[],3\right),\ \mathcal{Y}=\mathtt{ifft}\left(\bar{\mathcal{Y}},[],3\right).
	\end{equation*}
It is easy to verify that the tensors $\mathcal{X}$ and $\mathcal{Y}$ satisfy the requirements.
\end{proof}


\begin{theorem}\label{SMW_MOORE}(SMW formula for the tensor Moore-Penrose inverse).
	Suppose $\mathcal{A}\in \mathbb{R}^{n_1\times n_1 \times n_3}$, $\mathcal{U}\in \mathbb{R}^{n_1\times n_2 \times n_3}$, $\mathcal{V}\in \mathbb{R}^{n_2\times n_1 \times n_3}$, $\mathcal{B}\in \mathbb{R}^{n_2\times n_2 \times n_3}$, and the following conditions are satisfied:
	
	1. $\mathcal{U}=\mathcal{X}_1+\mathcal{Y}_1$, where the lateral slices of ${\mathcal{X}}_1$ are contained in $ \mathrm{R}(\mathcal{A})$, and the lateral slices of $\mathcal{Y}_1$ are orthogonal to $\mathrm{R}(\mathcal{A})$;
	
	2. $\mathcal{V}=\mathcal{X}_2+\mathcal{Y}_2$, where the lateral slices of ${\mathcal{X}}_2$ are contained in $ \mathrm{R}(\mathcal{A}^{\mathrm{T}})$, and the lateral slices of $\mathcal{Y}_1$ are orthogonal to $\mathrm{R}(\mathcal{A}^{\mathrm{T}})$;
	
	3. (1) $\mathcal{E}_2*\mathcal{B}^{\dagger}*\mathcal{E}_1^{\mathrm{T}}*\mathcal{Y}_1*\mathcal{B}=\mathcal{E}_2$
	(2) $\mathcal{X}_1*\mathcal{E}_1^T*\mathcal{Y}_1*\mathcal{B}=\mathcal{X}_1*\mathcal{B}$
	(3) $\mathcal{Y}_1*\mathcal{E}_1^T*\mathcal{Y}_1=\mathcal{Y}_1$;
	
	4. (1) $\mathcal{B}*\mathcal{Y}_2^T*\mathcal{E}_2*\mathcal{B}^{\dagger}*\mathcal{E}_1^T=\mathcal{E}_1^{\mathrm{T}}$
	(2) $\mathcal{B}*\mathcal{Y}_2^{\mathrm{T}}*\mathcal{E}_2*\mathcal{X}_2^{\mathrm{T}}=\mathcal{B}*\mathcal{X}_2^{\mathrm{T}}$
	(3) $\mathcal{E}_2*\mathcal{Y}_2^{\mathrm{T}}*\mathcal{E}_2=\mathcal{E}_2,$

	where $\mathcal{E}_i=\mathcal{Y}_i*(\mathcal{Y}_i^{\mathrm{T}}*\mathcal{Y}_i)^{\dagger},i=1,2.$
	
	Then the Moore-Penrose inverse of tensor
	\begin{equation*}
		\begin{aligned}
			\mathcal{M}&=\mathcal{A}+\mathcal{U}*\mathcal{B}*\mathcal{V}\\
			&=\mathcal{A}+(\mathcal{X}_1+\mathcal{Y}_1)*\mathcal{B}*(\mathcal{X}_2+\mathcal{Y}_2)^{\mathrm{T}}
		\end{aligned}
	\end{equation*}
	can be represented as:
	\begin{equation*}
		\mathcal{M}^{\dagger}=\mathcal{A}^{\dagger}-\mathcal{E}_2*\mathcal{X}_2^{\mathrm{T}}*\mathcal{A}^{\dagger}-\mathcal{A}^{\dagger}*\mathcal{X}_1*\mathcal{E}_1^{\mathrm{T}}+\mathcal{E}_2*(\mathcal{B}^{\dagger}+\mathcal{X}_2^{\mathrm{T}}*\mathcal{A}^{\dagger}*\mathcal{X}_1)*\mathcal{E}_1^{\mathrm{T}}.
	\end{equation*}
\end{theorem}
	\begin{proof}
        The formula is verified by direct computation.
		We recall that the Moore-Penrose inverse is the unique solution which satisfies the following four conditions:
		
		$(a) \mathcal{M}*\mathcal{M}^{\dagger}*\mathcal{M}=\mathcal{M},$~~~
		$(b) \mathcal{M}^{\dagger}*\mathcal{M}*\mathcal{M}^{\dagger}=\mathcal{M}^{\dagger},$\\
		$(c) (\mathcal{M}*\mathcal{M}^{\dagger})^{\mathrm{T}}=\mathcal{M}*\mathcal{M}^{\dagger},$~~~
		$(d) (\mathcal{M}^{\dagger}*\mathcal{M})^{\mathrm{T}}=\mathcal{M}^{\dagger}*\mathcal{M}.$
		
		 By simple expansion,
		\begin{equation}\label{4.4}
			\begin{aligned}
				\mathcal{M}*\mathcal{M}^{\dagger}=&\mathcal{A}*\mathcal{M}^{\dagger}-\mathcal{A}*\mathcal{A}^{\dagger}*\mathcal{X}_1*\mathcal{E}_1^{\mathrm{T}}+\mathcal{A}*\mathcal{E}_2*(\mathcal{B}^{\dagger}+\mathcal{X}_2^{\mathrm{T}}*\mathcal{A}^{\dagger}*\mathcal{X}_1)*\mathcal{E}_1\\
				&+(\mathcal{X}_1+\mathcal{Y}_1)*\mathcal{B}*(\mathcal{X}_2*+\mathcal{Y}_2)^{\mathrm{T}}*\mathcal{A}^{\dagger}\\
				&-(\mathcal{X}_1+\mathcal{Y}_1)*\mathcal{B}*(\mathcal{X}_2+\mathcal{Y}_2)^{\mathrm{T}}*\mathcal{E}_2*\mathcal{X}_2^{\mathrm{T}}*\mathcal{A}^{\dagger}\\
				&-(\mathcal{X}_1+\mathcal{Y}_1)*\mathcal{B}*(\mathcal{X}_2*+\mathcal{Y}_2)^{\mathrm{T}}*\mathcal{A}^{\dagger}*\mathcal{X}_1*\mathcal{E}_1^{\mathrm{T}}\\
				&+(\mathcal{X}_1+\mathcal{Y}_1)*\mathcal{B}*(\mathcal{X}_2*+\mathcal{Y}_2)^{\mathrm{T}}*\mathcal{E}_2*\mathcal{X}_2^{\mathrm{T}}*\mathcal{A}^{\dagger}*\mathcal{X}_1*\mathcal{E}_1^{\mathrm{T}}\\
				&+(\mathcal{X}_1+\mathcal{Y}_1)*\mathcal{B}*(\mathcal{X}_2*+\mathcal{Y}_2)^{\mathrm{T}}*\mathcal{E}_2*\mathcal{B}^{\dagger}*\mathcal{E}_1^{\mathrm{T}}.
			\end{aligned}
		\end{equation}
		Since the lateral slices of $\mathcal{Y}_2$ are orthogonal to $\mathrm{R}(\mathcal{A}^{\mathrm{T}})$, we have $\mathcal{A}*\mathcal{Y}_2=\mathbf{0}$, $\mathcal{Y}_2*\mathcal{A}^{\dagger}=\mathbf{0}$, and $\mathcal{X}_2^{\mathrm{T}}*\mathcal{Y}_2=\mathbf{0}$, which simplify (\ref{4.4}) to
		\begin{equation}\label{4.5}
			\begin{aligned}
				\mathcal{M}*\mathcal{M}^{\dagger}=&\mathcal{A}*\mathcal{A}^{\dagger}-\mathcal{A}*\mathcal{A}^{\dagger}*\mathcal{X}_1*\mathcal{E}_1^{\mathrm{T}}\\
				&+(\mathcal{X}_1+\mathcal{Y}_1)*\mathcal{B}*\mathcal{X}_2^{\mathrm{T}}*\mathcal{A}^{\dagger}
				-(\mathcal{X}_1+\mathcal{Y}_1)*\mathcal{B}*\mathcal{Y}_2^{\mathrm{T}}*\mathcal{E}_2*\mathcal{X}_2^{\mathrm{T}}*\mathcal{A}^{\dagger}\\
				&-(\mathcal{X}_1+\mathcal{Y}_1)*\mathcal{B}*\mathcal{X}_2^{\mathrm{T}}\mathcal{A}^{\dagger}*\mathcal{X}_1*\mathcal{E}_1^{\mathrm{T}}\\
				&+(\mathcal{X}_1+\mathcal{Y}_1)*\mathcal{B}*\mathcal{Y}_2^{\mathrm{T}}*\mathcal{E}_2*\mathcal{X}_2^{\mathrm{T}}*\mathcal{A}^{\dagger}*\mathcal{X}_1*\mathcal{E}_1^{\mathrm{T}}\\
				&+(\mathcal{X}_1+\mathcal{Y}_1)*\mathcal{B}*\mathcal{Y}_2^{\mathrm{T}}*\mathcal{E}_2*\mathcal{X}_2^{\mathrm{T}}*\mathcal{B}^{\dagger}*\mathcal{E}_1^{\mathrm{T}}.
			\end{aligned}
		\end{equation}
		Utilizing  $\mathcal{B}*\mathcal{Y}_2^{\mathrm{T}}*\mathcal{E}_2*\mathcal{B}^{\dagger}*\mathcal{E}_1^{\mathrm{T}}=\mathcal{E}_1^{\mathrm{T}}$,  $\mathcal{B}*\mathcal{Y}_2^{\mathrm{T}}*\mathcal{E}_2*\mathcal{X}_2^{\mathrm{T}}=\mathcal{B}*\mathcal{X}_2^{\mathrm{T}}$, and  $\mathcal{A}*\mathcal{A}^{\dagger}*\mathcal{X}_1=\mathcal{X}_1$, (\ref{4.5}) can be further reduced to
		\begin{equation*}
			\mathcal{M}*\mathcal{M}^{\dagger}=\mathcal{A}*\mathcal{A}^{\dagger}+\mathcal{Y}_1*\mathcal{E}_1^{\mathrm{T}}.
		\end{equation*}
		Therefore,
		\begin{equation*}
			\begin{aligned}
				(\mathcal{M}*\mathcal{M}^{\dagger})^{\mathrm{T}}&=(\mathcal{A}*\mathcal{A}^{\dagger}+\mathcal{Y}_1*\mathcal{E}_1^{\mathrm{T}})^{\mathrm{T}}\\
				&=(\mathcal{A}*\mathcal{A}^{\dagger})^{\mathrm{T}}+(\mathcal{Y}_1*\mathcal{E}_1^{\mathrm{T}})^{\mathrm{T}}=\mathcal{M}*\mathcal{M}^{\dagger},
			\end{aligned}
		\end{equation*}
		which implies (c).		
By mimicking the above argument, we obtain that
		\begin{equation*}
			(\mathcal{M}^{\dagger}*\mathcal{M})^{\mathrm{T}}=\mathcal{A}^{\dagger}*\mathcal{A}+\mathcal{E}_2*\mathcal{Y}_2^{\mathrm{T}},
		\end{equation*}
		which shows that (d) is valid.
		
		Moreover, we can get that
		\begin{equation}\label{4.6}
			\begin{aligned}
				\mathcal{M}*\mathcal{M}^{\dagger}*\mathcal{M}=&(\mathcal{A}*\mathcal{A}^{\dagger}+\mathcal{Y}_1*\mathcal{E}_1^{\mathrm{T}})*(\mathcal{A}+(\mathcal{X}_1+\mathcal{Y}_1)*\mathcal{B}*(\mathcal{X}_2+\mathcal{Y}_2)^{\mathrm{T}})\\
				=&\mathcal{A}*\mathcal{A}^{\dagger}*\mathcal{A}+\mathcal{Y}_1*\mathcal{E}_1^{\mathrm{T}}*\mathcal{A}\\
				&+ \mathcal{A}*\mathcal{A}^{\dagger}*(\mathcal{A}+(\mathcal{X}_1+\mathcal{Y}_1)*\mathcal{B}*(\mathcal{X}_2+\mathcal{Y}_2)^{\mathrm{T}})\\
				&+\mathcal{Y}_1*\mathcal{E}_1^{\mathrm{T}}*(\mathcal{A}+(\mathcal{X}_1+\mathcal{Y}_1)*\mathcal{B}*(\mathcal{X}_2+\mathcal{Y}_2)^{\mathrm{T}})\\
				=&\mathcal{A}+\mathcal{Y}_1*\mathcal{E}_1^{\mathrm{T}}*\mathcal{A}+\mathcal{A}*\mathcal{A}^{\dagger}*\mathcal{X}_1*\mathcal{B}*(\mathcal{X}_2+\mathcal{Y}_2)^{\mathrm{T}}\\
				&+\mathcal{A}*\mathcal{A}^{\dagger}*\mathcal{Y}_1*\mathcal{B}*(\mathcal{X}_2+\mathcal{Y}_2)^{\mathrm{T}}+\mathcal{Y}_1*\mathcal{E}_1^{\mathrm{T}}*\mathcal{X}_1*\mathcal{B}*(\mathcal{X}_2+\mathcal{Y}_2)^{\mathrm{T}}\\
				&+\mathcal{Y}_1*\mathcal{E}_1^T*\mathcal{Y}_1*\mathcal{B}*(\mathcal{X}_2+\mathcal{Y}_2)^{\mathrm{T}}.\\
			\end{aligned}
		\end{equation}
		Applying  $\mathcal{A}*\mathcal{A}^{\dagger}*\mathcal{X}_1=\mathcal{X}_1,\mathcal{A}^{\dagger}*\mathcal{Y}_1=0,\mathcal{Y}_1^{\mathrm{T}}*\mathcal{X}_1^{\mathrm{T}}=0$ and $\mathcal{Y}_1*\mathcal{E}_1^{\mathrm{T}}*\mathcal{Y}_1=\mathcal{Y}_1$, (\ref{4.6}) can be simplified as
		\begin{equation*}
			\begin{aligned}
				\mathcal{M}*\mathcal{M}^{\dagger}*\mathcal{M}&=\mathcal{A}+\mathcal{X}_1*\mathcal{B}*(\mathcal{X}_2+\mathcal{Y}_2)^{\mathrm{T}}+\mathcal{Y}_1*\mathcal{B}*(\mathcal{X}_2+\mathcal{Y}_2)^{\mathrm{T}}\\
				&=\mathcal{A}+(\mathcal{X}_1+\mathcal{Y}_1)*\mathcal{B}*(\mathcal{X}_2+\mathcal{Y}_2)^{\mathrm{T}}=\mathcal{M},
			\end{aligned}
		\end{equation*}
		which is precisely (a), and (b) can be proved analogously. 
		
		\hskip 2em Since all requirements for the  definition of the Moore-Penrose inverse have been fulfilled, the theorem is proved.
	\end{proof}

\hskip 2em Then we give a related lemma necessary for our derivation. According to \cite[Theorem 3.3]{Jin2017}, it is straightforward to obtain the following lemma.
\begin{lemma}\label{SMW_JIN}
	Suppose $\mathcal{A} \in \mathbb{R}^{n_1\times n_2\times n_3}$, $\mathcal{D}\in \mathbb{R}^{n_1\times n_4\times n_3}$ and $\mathcal{X} \in \mathbb{R}^{n_2\times n_4\times n_3}$. Then the tensor equation
	\begin{equation*}
		\mathcal{A}*\mathcal{X}=\mathcal{D},
	\end{equation*}
	has a solution if and only if $\mathcal{A}*\mathcal{A}^{\dagger}*\mathcal{D}=\mathcal{D}$. The solution can be expressed as
	\begin{equation*}
		\mathcal{X}=\mathcal{A}^{\dagger}*\mathcal{D}+(\mathcal{I}-\mathcal{A}^{\dagger}*\mathcal{A})*\mathcal{Y},
	\end{equation*}
	where $\mathcal{Y}\in\mathbb{R}^{n_1\times n_2\times n_3}$ is an arbitrary tensor.
\end{lemma}

\hskip 2em Now we are ready to apply the SMW formula to perform the sensitivity analysis
for a multilinear system of equations $\mathcal{A}*\mathcal{X}=\mathcal{D}$, and we derive
upper bounds for the error in the solution when the coefficient tensor and the right-hand side are perturbed.
\begin{theorem}
	Suppose the multilinear system of equations is
	\begin{equation*}
		\mathcal{A}*\mathcal{X}=\mathcal{D}
	\end{equation*}
	where $\mathcal{A} \in \mathbb{R}^{n_1\times n_2\times n_3},\mathcal{X} \in \mathbb{R}^{n_2\times n_4\times n_3}$ and $\mathbf{0}\neq\mathcal{D} \in \mathbb{R}^{n_1\times n_4\times n_3}$. The perturbed system is expressed as
	\begin{equation*}
		(\mathcal{A}+\mathcal{E})*\mathcal{Y}=(\mathcal{D}+\mathcal{H}),
	\end{equation*}
where $\mathcal{E}\in \mathbb{R}^{n_1\times n_2\times n_3}$ and $\mathcal{H} \in \mathbb{R}^{n_1\times n_4\times n_3}$. If the tensor $\mathcal{E}$ is decomposed as
	\begin{equation*}
		\mathcal{E}=\mathcal{U}*\mathcal{B}*\mathcal{V}=(\mathcal{X}_1+\mathcal{Y}_1)*\mathcal{B}*(\mathcal{X}_2+\mathcal{Y}_2)^T,
	\end{equation*}
	where the lateral slices of $\mathcal{X}_1$, $\mathcal{X}_2$ are contained in $ \mathrm{R}(\mathcal{A})$ and $\mathrm{R}(\mathcal{A}^{\mathrm{T}})$ respectively and the lateral slices of $\mathcal{Y}_1$, $\mathcal{Y}_2$ are orthogonal to $\mathrm{R}(\mathcal{A})$ and  $\mathrm{R}(\mathcal{X}^{\mathrm{T}})$ severally.
	
	\hskip 2em We further assume that $\|\mathcal{X}_i\|\le\epsilon_A\|\mathcal{A}\|$, $\|\mathcal{E}_i\|\le\epsilon_A\|\mathcal{A}\|$, $\|\mathcal{B}^{\dagger}\|\le\epsilon_A\|\mathcal{A}\|$ and $\|\mathcal{H}\|\le\epsilon_D\|\mathcal{D}\|$ for $1\le i \le 2$, then
	\begin{equation}\label{app1}
		\begin{aligned}
			\frac{\|\mathcal{Y}-\mathcal{X}\|_F}{\|\mathcal{X}\|_F}\le&(1+\epsilon_D)\|\mathcal{D}\|^3_F\|\mathcal{X}\|_F\left(2n_3^2\epsilon_A^2\mathcal{A}^{\dagger}+n_3^2\epsilon_A^3\|\mathcal{A}\|_F+n_3^3\epsilon_A^4\|\mathcal{A}\|_F^2\|\mathcal{A}^{\dagger}\|_F\right)\\&+n_3\epsilon_D\|\mathcal{A}\|_F\|\mathcal{A}^{\dagger}\|_F,
		\end{aligned}
	\end{equation}
	and
	\begin{equation*}
		\begin{aligned}
			\frac{\|\mathcal{Y}-\mathcal{X}\|_2}{\|\mathcal{X}\|_2}\le&(1+\epsilon_D)\|\mathcal{D}\|^3_2\|\mathcal{X}\|_2\left(2\epsilon_A^2\mathcal{A}^{\dagger}+\epsilon_A^3\|\mathcal{A}\|_2+\epsilon_A^4\|\mathcal{A}\|_2^2\|\mathcal{A}^{\dagger}\|_2\right)\\&+\epsilon_D\|\mathcal{A}\|_2\|\mathcal{A}^{\dagger}\|_2.
		\end{aligned}
	\end{equation*}
\end{theorem}
\begin{proof}
	We only consider the Frobenius norm case.
	It follows from Lemma \ref{SMW_JIN} that $\mathcal{X}$ can be expressed as
	\begin{equation*}
		\mathcal{X}=\mathcal{A}^{\dagger}*\mathcal{D}+(\mathcal{I}-\mathcal{A}^{\dagger}*\mathcal{A})*\mathcal{U}.
	\end{equation*}
	Analogously, $\mathcal{Y}$ can be expressed as
	\begin{equation*}
		\mathcal{Y}=(\mathcal{A+E})^{\dagger}*(\mathcal{D}+\mathcal{H})+(\mathcal{I}-(\mathcal{A}+\mathcal{E})^{\dagger}*(\mathcal{A+E}))*\mathcal{U}.
	\end{equation*}
Particularly, if we set $\mathcal{U}=\mathbf{0}$, combining with Theorem \ref{SMW_MOORE}, the following equation holds.
	\begin{equation*}
		\begin{aligned}
			\mathcal{Y}-\mathcal{X}=&(\mathcal{A}+\mathcal{E})^{\dagger}*(\mathcal{D}+\mathcal{H})-\mathcal{A}^{\dagger}*\mathcal{D}\\
			=&[\mathcal{A}+(\mathcal{X}_1+\mathcal{Y}_1)*\mathcal{B}*(\mathcal{X}_2+\mathcal{Y}_2)^{\mathrm{T}}]^{\dagger}*(\mathcal{D}+\mathcal{H})-\mathcal{A}^{\dagger}*\mathcal{D}\\
			=&(\mathcal{A}^{\dagger}-\mathcal{E}_2*\mathcal{X}_2^{\mathrm{T}}*\mathcal{A}^{\dagger}-\mathcal{A}^{\dagger}*\mathcal{X}_1*\mathcal{E}_1^{\mathrm{T}}+\mathcal{E}_2*(\mathcal{B}^{\dagger}+\mathcal{X}_2^{\mathrm{T}}*\mathcal{A}^{\dagger}*\mathcal{X}_1)*\mathcal{E}_1^{\mathrm{T}})*\\&(\mathcal{D}+\mathcal{H})-\mathcal{A}^{\dagger}*\mathcal{D}\\
			=&\mathcal{E}_2*\mathcal{X}_2^{\mathrm{T}}*\mathcal{A}^{\dagger}*\mathcal{D}-\mathcal{A}^{\dagger}*\mathcal{X}_1*\mathcal{E}_1*\mathcal{D}+\mathcal{E}_2*(\mathcal{B}^{\dagger}+\mathcal{X}_2^{\mathrm{T}}*\mathcal{A}^{\dagger}*\mathcal{X}_1)*\mathcal{E}_1^{\mathrm{T}}*\mathcal{D}+\\
			&\mathcal{A}^{\dagger}*\mathcal{H}-\mathcal{E}_2*\mathcal{X}_2^{\mathrm{T}}*\mathcal{A}^{\dagger}*\mathcal{D}^{\dagger}-\mathcal{A}^{\dagger}*\mathcal{X}_1)*\mathcal{E}_1^{\mathrm{T}}*\mathcal{H}+\\
			&\mathcal{E}_2*(\mathcal{B}^{\dagger}+\mathcal{X}_2^{\mathrm{T}}*\mathcal{A}^{\dagger}*\mathcal{X}_1)*\mathcal{E}_1^{\mathrm{T}}*\mathcal{H}.
		\end{aligned}
	\end{equation*}
	Taking the Frobenius norm on both sides of the relation above and utilizing the properties of the tensor norm, we obtain
	\begin{equation*}
		\begin{aligned}		\|\mathcal{Y}-\mathcal{X}\|_F\le&n_3^{3/2}\|\mathcal{E}_2\|_F\|\mathcal{X}_2^\mathrm{T}\|_F\|\mathcal{A}^{\dagger}\|_F\|\mathcal{D}\|_F+n_3^{3/2}\|\mathcal{A}^{\dagger}\|_F\|\mathcal{X}_1\|_F\|\mathcal{E}_1^\mathrm{T}\|_F\|\mathcal{D}\|_F+\\
			&n_3^{3/2}\|\mathcal{E}_2\|_F\|\mathcal{B}^{\dagger}\|_F\|\mathcal{E}_1^{\mathrm{T}}\|_F\|\mathcal{D}\|_F+n_3^2\|\mathcal{X}^\mathrm{T}\|_F\|\mathcal{A}^{\dagger}\|_F\|\mathcal{X}_1\|_F\|\mathcal{E}_1^\mathrm{T}\|_F\|\mathcal{D}\|_F+\\	&n_3^{1/2}\|\mathcal{A}^{\dagger}\|_F\|\mathcal{H}\|_F+n_3^{3/2}\|\mathcal{E}_2\|_F\|\mathcal{X}_2^{\mathrm{T}}\|_F\|\mathcal{A}^{\dagger}\|_F\|\mathcal{H}\|_F+\\
			&n_3^{3/2}\|\mathcal{A}^{\dagger}\|_F\|\mathcal{X}_1\|_F\|\mathcal{E}_1^{\mathrm{T}}\|_F\|\mathcal{H}\|_F+\|n_3^{3/2}\mathcal{E}_2\|_F\|\mathcal{B}^{\dagger}\|_F\|\mathcal{E}_1^{\dagger}\|_F\|\mathcal{H}\|_F+\\
			&n_3^2\|\mathcal{X}_2^{\mathrm{T}}\|_F\|\mathcal{A}^{\dagger}\|_F\|\mathcal{X}_1\|_F\|\mathcal{E}_1^{\mathrm{T}}\|_F\|\mathcal{H}\|_F.
		\end{aligned}
	\end{equation*}
	Recall that $\|\mathcal{X}_i\|_F\le\epsilon_A\|\mathcal{A}\|_F$, $\|\mathcal{E}_i\|_F\le\epsilon_A\|\mathcal{A}\|_F$, $\|\mathcal{B}^{\dagger}\|_F\le\epsilon_A\|\mathcal{A}\|_F$, and  $\|\mathcal{H}\|_F\le\epsilon_D\|\mathcal{D}\|_F$, then
	\begin{equation*}
		\begin{aligned}
			\|\mathcal{Y}-\mathcal{X}\|_F\le&(1+\epsilon_D)\|\mathcal{D}\|_F(2n_3^{3/2}\epsilon_A^2\|\mathcal{A}\|_F^2\|\mathcal{A}^{\dagger}\|_F^2+n_3^{3/2}\epsilon_A^3\|\mathcal{A}\|_F^3\|_F+\\&n_3^{5/2}\epsilon_A^4\|\mathcal{A}\|_F^4\|\mathcal{A}^{\dagger}\|_F)+n_3^{1/2}\|\mathcal{A}^{\dagger}\|_F\|\mathcal{D}\|_F.
		\end{aligned}
	\end{equation*}
	Since $\|\mathcal{D}\|_F=\|\mathcal{A}*\mathcal{X}\|_F\le n_3^{1/2}\|\mathcal{A}\|_F\|\mathcal{X}\|_F$, the inequality can be further reduced to
	\begin{equation*}
		\begin{aligned}
			\frac{\|\mathcal{Y}-\mathcal{X}\|_F}{\|\mathcal{X}\|_F}\le& (1+\epsilon_D)(2n_3^2\epsilon_A^2\|\mathcal{A}^{\dagger}\|_F+n_3^2\epsilon_A^3\|\mathcal{A}\|_F+n_3^3\epsilon_A^4\|\mathcal{A}\|_F^2\|\mathcal{A}^{\dagger}\|_F)+\\&n_3\epsilon_D\|\mathcal{A}^{\dagger}\|_F\|\mathcal{A}\|_F,
		\end{aligned}
	\end{equation*}
	which leads to (\ref{app1}).
\end{proof}	
	\section{Conclusions}
	\hskip 2em In this paper, we give perturbation theorems of the tensor inverses and tensor equations.
     Additionally, we present a collection of perturbation results for the tensor Moore-Penrose inverse and the tensor least squares problem. Meanwhile, the sensitivity analysis for the multilinear system of equations is also performed by the extension of the SMW formula.
	
	\hskip 2em  
	In contrast to the Moore-Penrose inverse, the Drazin inverse based on the t-product has already been considered. 
	Naturally, we may consider the corresponding perturbation in the future.
	
    \section*{Acknowledgments}
	\hskip 2em
        This work is supported by the National Natural Science Foundation of China under grant 11801534.

	\bibliographystyle{siam}
	\bibliography{perturbation}
\end{document}